\documentclass[12pt,english]{amsart}

\usepackage{amsmath}
\usepackage{amsthm}
\usepackage{amssymb}
\usepackage{newpxmath, newpxtext}
\usepackage{enumerate}
\usepackage{caption}
\usepackage[pdftex]{graphicx} 
\usepackage{bm}
\usepackage{mathrsfs} 
\usepackage[svgnames]{xcolor} 
\usepackage{tikz} 
\usepackage{multicol}
\usepackage{here} 
\allowdisplaybreaks 
\usepackage{comment} 
\usepackage[absolute,overlay]{textpos} 
\usepackage{calc}

\newcommand{\Z}{\mathbb{Z}}
\newcommand{\R}{\mathbb{R}}
\newcommand{\F}{\mathbb{F}}
\newcommand{\Sc}{\mathcal{S}}
\newcommand{\Tc}{\mathcal{T}}
\newcommand{\Gc}{\mathcal{G}}
\theoremstyle{theorem}
\newtheorem{thm}{Theorem}
\newtheorem{lem}{Lemma}
\newtheorem{prop}{Proposition}
\newtheorem{cor}{Corollary}
\newtheorem*{proof*}{Proof}
\theoremstyle{definition}
\newtheorem{ex}{Example}

\title[Adjacency matrices over $\F_p$]{Adjacency matrices over a finite prime field and their direct sum decompositions}
\author{Akihiro Higashitani}
\address{Department of Pure and Applied Mathematics, Graduate School of Information Science and Technology, Osaka University}
\email{higashitani@ist.osaka-u.ac.jp}
\author{Yuya Sugishita}
\address{Department of Pure and Applied Mathematics, Graduate School of Information Science and Technology, Osaka University}
\email{u977762e@ecs.osaka-u.ac.jp}

\begin{document}
\maketitle
\begin{abstract}
In this paper, we discuss the adjacency matrices of finite undirected simple graphs over a finite prime field $\F_p$. 
We apply symmetric (row and column) elementary transformations to the adjacency matrix over $\F_p$ 
in order to get a direct sum decomposition by other adjacency matrices. 
In this paper, we give a complete description of the direct sum decomposition of the adjacency matrix of any graph over $\F_p$ for any odd prime $p$. 
Our key tool is quadratic residues of $\F_p$. 
\end{abstract}

\section{Introduction}
Graphs appearing in this paper are always assumed to be finite, undirected and simple. 

\subsection{Adjacency matrices}
Given a graph $X$ on the vertex set $V(G)$ with the edge set $E(G)$, 
the adjacency matrix $A(X)$ of $X$ is a $|V(G)| \times |V(G)|$ matrix defined by
$$A(X)_{uv}=\begin{cases}
1 \;\;\text{if }uv \in E(G), \\
0 \;\;\text{if }uv \not\in E(G). 
\end{cases}$$

Adjacency matrices of graphs and their spectrum often give a characterization of several properties on graphs. 
For example, in \cite{CEG}, a relationship between the smallest eigenvalues of the adjacency matrix and the graph structure is discussed. 
In \cite{MMMA}, several inequalities on the absolute values of the eigenvalues of the adjacency matrices and their determinants are shown. 
Moreover, given a graph $X$, let $S(X)=J-I-2A(X)$, where $I$ is the identity matrix and $J$ is the all-one matrix. 
The symmetric matrix $S(X)$ is called the \textit{Seidel matrix} of a graph $X$. 
This has some connections with other combinatorial objects. 
For example, the eigenvalues of Seidel matrices are used for the investigations or the characterizations of equiangular lines and strongly regular graphs. 
For more information, please consult, e.g., \cite{GR}. 

On these studies, adjacency matrices and Seidel matrices are treated as the matrices over $\R$. 
On the other hand, there are little studies on the adjacency matrices (or Seidel matrices) over a finite field as far as the authors know. 
The goal of this paper is to initiate the studies on adjacency matrices over a finite prime field.

\subsection{Prime fields and quadratic residues}
For a prime $p$, let $\F_p$ denote the prime field of order $p$ and let $\F_p^\times = \F_p \setminus \{0\}$. 
For $\F_p$, let $\Sc(p)$ (resp. $\Tc(p)$) be the set of quadratic residues (resp. quadratic nonresidues) of $\F_p$. Namely,  
$$\Sc(p)=\{a^2 : a \in \F_p^\times\}\;\text{ and }\;\Tc(p)=\F_p^\times \setminus \Sc(p).$$ 
In particular, $\F_p^\times=\mathcal{S}(p)\sqcup\mathcal{T}(p)$.
For example, we list the quadratic (non)residues of $\F_p$ for some small odd primes: 
\begin{table}[h]
\centering
\begin{tabular}{|c||c|c|} \hline
$p$ &$\Sc(p)$ &$\Tc(p)$ \\ \hline\hline
$3$ &1 &$-1$ \\ \hline
$5$ &$\pm 1$ &$\pm 2$ \\ \hline
$7$ &$1,2,-3$ &$-1,-2,3$ \\ \hline
$11$ &$1,-2,3,4,5$ &$-1,2,-3,-4,-5$ \\ \hline 
$13$ &$\pm 1, \pm 3, \pm 4$ &$\pm 2, \pm 5, \pm 6$ \\ \hline
\end{tabular}
\caption{Examples of quadratic (non)residues}
\end{table}

Regarding quadratic residues of $\F_p$, the following facts are well known: 
\begin{itemize}
\item If $p \geq 3$, then $|\Sc(p)|=|\Tc(p)|=(p-1)/2$. 
\item For $x,y \in \F_p^\times$, if $x,y \in \Sc(p)$ or $x,y \in \Tc(p)$, then $xy \in \Sc(p)$. 
\item If $x \in \Sc(p)$ and $y \in \Tc(p)$, then $xy \in \Tc(p)$. 
\item $x \in \Sc(p)$ if and only if $x^{-1} \in \Sc(p)$. 
\end{itemize}
For more details, please consult, e.g., \cite{HW}. 

\subsection{Direct sum decompositions of symmetric matrices}
Given two square matrices $M$ and $M'$, let $M \oplus M'$ denote the direct sum of $M$ and $M'$, i.e., 
$M \oplus M'=\begin{pmatrix} M&\\ &M'\\ \end{pmatrix}$. 
Moreover, we use the notation $nM:=\underbrace{M\oplus\cdots\oplus M}_{n}$. 

Let $k$ be a field. 
For symmetric matrices $M$ and $M'$ whose entries belong to $k$, we say that $M$ and $M'$ are \textit{similar} over $k$ 
if there is a regular matrix $P$ over $k$ such that $M'=P^t MP$, where $P^t$ denotes the transpose of $P$. 
We use the notation like $M \sim M'$ if $M$ and $M'$ are similar, or $M \underset{P}{\longrightarrow} M'$ if $M'=P^t MP$. 
It is well known that any symmetric matrix is similar over $\R$ to a certain diagonal matrix. 

\bigskip

In what follows, we abuse the notation for the adjacency matrix of $X$ as the same symbol of a graph $X$. 

Given a graph $X$, we say that \textit{$X$ can be decomposed into $X_1,\ldots,X_s$} over a field $k$, where $X_1,\ldots,X_s$ are graphs, 
if $X$ is similar over $k$ to a direct sum of $X_i$'s, i.e., $X \sim n_1X_1 \oplus \cdots \oplus n_sX_s$ for some $n_i \in \Z_{\geq 0}$. 

In \cite[Section 8.10]{GR}, the rank of the adjacency matrix over $\F_2$ is studied. 
By using the discussions there, we can claim the following: 
\begin{thm}[{cf. \cite[Section 8.10]{GR}}]\label{p=2}
Any graph $X$ can be decomposed into $K_1$ and $K_2$ over $\F_2$. 
\end{thm}
Here, $K_n$ denotes the complete graph on $n$ vertices

The goal of this paper is to develop the similar result to Theorem~\ref{p=2} in the case of odd primes.

\subsection{Main Results}

We divide the statements of our main theorems into six cases as shown in the table below.

\begin{table}[h]
\centering
\begin{tabular}{|c|c|c|c|} \hline
&$-1,2\in\mathcal{S}(p)$&$-1\in\mathcal{S}(p),2\in\mathcal{T}(p)$&$-1\in\mathcal{T}(p)$\\ \hline
$3\in\mathcal{S}(p)$&Theorem \ref{allins}&Theorem \ref{only2int} &Theorem \ref{mainthm} (3)\\ \hline
$3\in\mathcal{T}(p)$&Theorem \ref{mainthm} (1)&Theorem \ref{mainthm}  (2)&Theorem \ref{mainthm} (4)\\ \hline
\end{tabular}
\caption{Division by six cases of main theorems}\label{tab:cases}
\end{table}
For example, the primes satisfying each condition are as follows: 
\begin{table}[h]
\centering
\begin{tabular}{|c|c|c|c|} \hline
&$-1,2\in\mathcal{S}(p)$&$-1\in\mathcal{S}(p),2\in\mathcal{T}(p)$&$-1\in\mathcal{T}(p)$\\ \hline
$3\in\mathcal{S}(p)$&73&13 &11 \\ \hline
$3\in\mathcal{T}(p)$&7,17 &5&19 \\ \hline
\end{tabular}
\caption{Examples of primes of Table~\ref{tab:cases}}
\end{table}

\begin{thm}\label{mainthm} Let $p$ be an odd prime. Then the following assertions hold: 
\begin{enumerate}
\item If $-1,2\in \Sc(p)$ and $3\in \Tc(p)$, then any graph can be decomposed into $K_1$, $K_2$, $K_3$ and $K_4$ over $\F_p$. 
\item If $-1\in \Sc(p)$ and $2,3\in \Tc(p)$, then any graph can be decomposed into $K_1$, $K_2$, $K_3$, $K_4$ and $B$ over $\F_p$. 
\item If $p \geq 5$, $3\in \Sc(p)$ and $-1\in \Tc(p)$, then any graph can be decomposed into $K_1$, $K_2$, $K_3$, $K_4$ and $C_5$ over $\F_p$. 
\item If $p \geq 5$ with $-1,3 \in \Tc(p)$ or $p=3$, then then any graph can be decomposed into $K_1$, $K_2$, $K_3$ and $C_5$ over $\F_p$. 
\end{enumerate}
\end{thm}

Here, $C_5$ denotes the cycle of length $5$ and $B$ denotes the graph with $5$ vertices depicted in Figure~\ref{fig:B}.  

\begin{thm}\label{only2int}
Let $p \geq 5$ and assume that $-1,3\in \Sc(p)$ and $2\in \Tc(p)$. 
Then any graph can be decomposed into $K_1$, $K_2$, $K_3$, $B$ and $X_6$ over $\F_p$, 
where $X_6:=K_6$ if $5\in \Tc(p)$, $X_6:=D$ if $7\in \Tc(p)$, and $X_6$ is not required otherwise. 
\end{thm}

Here, $D$ is the following graph with $6$ vertices (Figure~\ref{fig:D}).

\begin{figure}[h]
		\begin{minipage}[b]{0.47\textwidth}
			\centering
			\begin{tikzpicture}[scale=2.0]
				\coordinate (O)at(0,0);
				\foreach \i in {45,135,225,315}{
					\coordinate (\i)at(\i:1);
					\draw(O)--(\i);
				}
				\draw(45)--(315);
				\draw(135)--(225);
				\foreach \i in {O,45,135,225,315}
				\filldraw[fill=white] (\i) circle(0.1);
			\end{tikzpicture}
			\caption{Graph $B$}\label{fig:B}
		\end{minipage}
		\hfil
		\begin{minipage}[b]{0.48\textwidth}
			\centering
			\begin{tikzpicture}[scale=2.0]
				\coordinate(A)at(1/2,0);
				\coordinate(B)at(0,0.6);
				\coordinate(C)at(-1/2,0);
				\coordinate(D)at(0,-0.6);
				\coordinate(E)at(1.3,0);
				\coordinate(F)at(-1.3,0);
				\draw(A)--(B)--(C)--(D)--cycle;
				\draw(B)--(D);
				\foreach \i in {A,B,D}
				\draw(E)--(\i);
				\foreach \i in {B,C,D}
				\draw(F)--(\i);
				\foreach \i in {A,B,C,D,E,F}
				\filldraw[fill=white] (\i) circle(0.1);
			\end{tikzpicture}
			\caption{Graph $D$}\label{fig:D}
		\end{minipage}
	\end{figure}

\begin{thm}\label{allins}
Let $p \geq 5$ and assume that $-1,2,3\in \Sc(p)$. 
Then there exist graphs $X_4$ and $X_5$ such that any graph can be decomposed into $K_1$, $K_2$, $K_3$, $X_4$ and $X_5$ over $\F_p$.
\end{thm}

\subsection{Structure of this paper}
A brief structure of this paper is as follows. 
In Section~\ref{sec:pre}, we prepare some materials and prove some lemmas for the proofs of our main theorems. 
In Section~\ref{sec:mainthm}, we prove Theorem~\ref{mainthm}. 
In Section~\ref{sec:only2int}, we prove Theorem~\ref{only2int}. 
In Section~\ref{sec:allins}, we prove Theorem~\ref{allins}.

\subsection*{Acknowledgements}
The authors would like to thank Prof. Akihiro Munemasa for his helpful comments on the regularization of quadratic forms over a finite prime field. 
The first named author is partially supported by JSPS Grantin-Aid for Scientists Research (C) JP20K03513.

\bigskip

\section{Preliminaries}\label{sec:pre}

Let $\textbf{Graph}(n)$ denote the set of all non-isomorphic graphs with $n$ vertices. 
Since we identify graphs with their adjacency matrices, 
$\textbf{Graph}(n)$ is also regarded as the set of all adjacency matrices of graphs with $n$ vertices up to row and column permutations. 

\begin{ex}\label{ex:small}
For small $n$'s, we see the following by direct computations: 
\begin{align*}
\textbf{Graph}(2)&=\{K_1 \oplus K_1, K_2 \}; \\
\textbf{Graph}(3)/\sim&=\{3K_1, K_1 \oplus K_2, K_3\}; \\ 
\textbf{Graph}(4)/\sim&=\{4K_1, 2K_1 \oplus K_2, K_2 \oplus K_2, K_1 \oplus K_3, K_4\}, 
\end{align*}
where ``$/\sim$'' stands for up to similarity. 
\end{ex}

\begin{prop}[cf.~{\cite[Proposition 2.2]{A}}]\label{prop:det}
Let 
\begin{align*}
\mathcal{G}_n:=\{\det X \in \R \mid X\in\emph{\textbf{Graph}}(n)\}\setminus\{0\}.
\end{align*}
Then we have  
\begin{align*}
\mathcal{G}_2&=\{-1\}, \;\; \mathcal{G}_3=\{2\}, \;\; \mathcal{G}_4=\{-3,1\}, \;\; \mathcal{G}_5=\{-4,2,4\}, \text{ and }\\ 
\mathcal{G}_6&=\{-5,-4,-1,3,4,7\}.
\end{align*}
\end{prop}

\begin{ex}\label{ex:det}
For the latter discussions, we list the determinants of some adjacency matrices: 
\begin{align*}
&\det K_2 = -1, \;\; \det K_3 = 2, \;\; \det K_4 = -3, \;\; \det K_6=-5, \\
&\det B = -4, \;\; \det C_5=2, \;\; \det D=7. 
\end{align*}
\end{ex}

\begin{lem}\label{lem:iso}
Let $X$ be a graph and assume that $\det X=0$. 
Then there is a vertex $v$ of $X$ such that $X \sim K_1 \oplus (X \setminus v)$. 
\end{lem}
\begin{proof}
When $\det X=0$, there is a row of $X$, say, the first row, which can be written as a linear combination of other rows. 
This means that we can make the first row the zero vector by applying a certain elementary row operation. 
Namely, there is a regular matrix $P$ such that the first row and column of $P^tXP$ are the zero vector. This implies that $X \sim K_1 \oplus (X \setminus v)$, 
where $v$ corresponds to the first row and column. 
\end{proof}

The following will play an important role in the proofs of our theorems. 
\begin{prop}\label{prop} Fix $\F_p$. Then, for any $x \in \F_p^\times$, the matrices $\begin{pmatrix}
1&0\\
0&1
\end{pmatrix}$ and $\begin{pmatrix}
x&0\\
0&x
\end{pmatrix}$ are similar over $\F_p$. 
\end{prop}
\begin{proof}
We divide the discussions into two cases: $x \in \Sc(p)$ or $x \in \Tc(p)$. 

Let $x \in \Sc(p)$. Then 
$\begin{pmatrix}1 & \\ &1\end{pmatrix} \underset{\begin{pmatrix} \sqrt{x} & \\ &\sqrt{x}\end{pmatrix}}{\longrightarrow} \begin{pmatrix} x & \\ &x \end{pmatrix}$. 

Let $x \in \Tc(p)$.

\noindent
{\bf Claim}: There are $a,b\in\mathcal{S}(p)$ such that $x=a+b$. 

\noindent
(Proof) Let $t:=\min\{y\in\mathcal{T}(p)\mid y\in\F_p, 1<y<p\}$.
Note that $t-1\in\mathcal{S}(p)$ by definition of $t$. Since $s:=xt^{-1} \in \Sc(p)$ for every $x\in\mathcal{T}(p)$, 
we have $\displaystyle x=st=\underbrace{s}_{\in \Sc(p)}+\underbrace{s(t-1)}_{\in \Sc(p)}.$ \qed 

By using a description $x=a+b$ with $a,b \in \Sc(p)$, we see that 
\begin{align*}
\begin{pmatrix}
1&\\
&1
\end{pmatrix}
\underset{\begin{pmatrix} \sqrt{a} &\sqrt{b} \\ -\sqrt{b} &\sqrt{a} \end{pmatrix}}{\longrightarrow}
\begin{pmatrix}
a+b&\\
&a+b
\end{pmatrix}=
\begin{pmatrix}
x&\\
&x
\end{pmatrix}.
\end{align*}
\end{proof}

By Proposition~\ref{prop}, we immediately obtain the following:  
\begin{cor}\label{kei}
Let $X$ be a graph and let $p$ be a prime. Then the following statements hold: 
\begin{enumerate}
\item $\det X \in \Sc(p)$ if and only if $X \sim \begin{pmatrix} 1 & & & \\ &\ddots & & \\ & &1 & \\ & & &1 \end{pmatrix}$ over $\F_p$;  
\item $\det X \in \Tc(p)$ if and only if $X \sim \begin{pmatrix} 1 & & & \\ &\ddots & & \\ & &1 & \\ & & &x \end{pmatrix}$ over $\F_p$, 
where $x$ is some element of $\Tc(p)$. 
\end{enumerate}
\end{cor}

\bigskip

\section{Proof of Theorem~\ref{mainthm}}\label{sec:mainthm}

This section is devoted to proving Theorem~\ref{mainthm}. 

Let $X$ be an arbitrary graph. We show the statements by induction on $n:=|V(X)|$. 
By Example~\ref{ex:small}, we see that the assertions (1), (2) and (3) hold for $n \leq 4$, while we can check (4) only for $n \leq 3$. 

An idea of our proof is as follows. 
Fix a vertex $v$ of $X$ and let $X'=X\setminus v$. If $n$ is sufficiently large, then we can decompose $X'$ into the certain graphs by the hypothesis of induction. 
If $\det(X')=0$, by Lemma~\ref{lem:iso}, we see that $X' \sim \ell K_1 \oplus X''$, where $\det(X'') \neq 0$. 
Thus, we may assume the following: 
$$X \sim \begin{pmatrix}0 & & & & & &0 \\ &\ddots & & & & &\vdots \\ & &0 & & & &0 \\ & & & & & &* \\ & & & &\text{\LARGE{$X''$}} & &\vdots \\
& & & & & &* \\ 0 &\cdots &0 &* &\cdots &* &\star  \end{pmatrix}, $$
where the left-upper part corresponds to $\ell$ copies of $K_1$ and we let $\det X'' \neq 0$. 
Note that the fisrt $\ell$ entries of the last colum (resp. row) can be assumed to be $0$; 
otherwise we can easily see that the right-hand side is similar to $(\ell-2)K_1 \oplus K_2 \oplus X''$.  

Moreover, since $\det(X'') \neq 0$, the rows of $X''$ are linearly independent, so we can make the last row all $0$ except for $\star$. 
Hence, we see that 
$$\begin{pmatrix} & & &* \\ &X'' & &\vdots \\ & & &* \\ * &\cdots &* &\star \end{pmatrix} \sim
\begin{pmatrix} & & &0 \\ &X'' & &\vdots \\ & & &0 \\ 0 &\cdots &0 &\clubsuit\end{pmatrix}=:\tilde{X}.$$ 

Our remaining task is to write $\tilde{X}$ as a direct sum of certain adjacency matrices. 
Note that $X''$ is already decomposed into certain adjacency matrices by the hypothesis of induction. 

We divide the discussions into two cases; either $\clubsuit \in \Sc(p)$ or $\clubsuit \in \Tc(p)$. 
Note that we are done in the case $\clubsuit=0$.

\bigskip

\noindent
\underline{(1) $-1,2\in\mathcal{S}(p)$ and $3\in\mathcal{T}(p)$:} 
In this case, we see the following: 
$$
K_2
\sim\begin{pmatrix}
1&\\
&1
\end{pmatrix}, \;
K_3
\sim\begin{pmatrix}
1&&\\
&1&\\
&&1
\end{pmatrix}, \text{ and }
K_4
\sim\begin{pmatrix}
1&&&\\
&1&&\\
&&1&\\
&&&3
\end{pmatrix} \text{ over }\F_p.$$ 
Note that these follow from Example~\ref{ex:det} and Corollary~\ref{kei}. 

Let $\clubsuit\in\mathcal{S}(p)$. 
\begin{itemize}
\item If $K_2$ appears in the direct summand of $X''$, since 
$\begin{pmatrix}
K_2&\\
&\clubsuit \\
\end{pmatrix}\sim 
\begin{pmatrix}
1&&\\
&1&\\
&&1\\
\end{pmatrix}\sim
K_3$, 
we obtain a decomposition of $\tilde{X}$. 
\item If $K_3$ appears, since $\begin{pmatrix}
K_3&\\
&\clubsuit
\end{pmatrix}\sim
\begin{pmatrix}
1&&&\\
&1&&\\
&&1&\\
&&&1 
\end{pmatrix}\sim
\begin{pmatrix}
K_2&\\
&K_2
\end{pmatrix}$, we obtain a decomposition of $\tilde{X}$. 
\item If two $K_4$'s appear, since 
\begin{align*}
\begin{pmatrix}
K_4&&\\
&K_4&\\
&&\clubsuit
\end{pmatrix}&\sim
\begin{pmatrix}
1&&&&&&&&\\
&1&&&&&&&\\
&&1&&&&&&\\
&&&3&&&&&\\
&&&&1&&&&\\
&&&&&1&&&\\
&&&&&&1&&\\
&&&&&&&3&\\
&&&&&&&&1
\end{pmatrix}\\
&\underset{\text{Proposition~\ref{prop}}}{\sim}\begin{pmatrix}
1&&&&&&&&\\
&1&&&&&&&\\
&&1&&&&&&\\
&&&1&&&&&\\
&&&&1&&&&\\
&&&&&1&&&\\
&&&&&&1&&\\
&&&&&&&1&\\
&&&&&&&&1
\end{pmatrix}
\sim\begin{pmatrix}
K_3&&\\
&K_3&\\
&&K_3
\end{pmatrix},\end{align*} we obtain a decomposition of $\tilde{X}$. 
\item We see that $X''\sim\begin{pmatrix} K_4&\\ &\clubsuit \end{pmatrix}$ never happens. 
In fact, if it happens, then $\det(X'') \in \Tc(p)$ by our assumption, while $\mathcal{G}_5=\{-4,2,4\}$ by Proposition~\ref{prop:det}
and $\{-4,2,4\} \subset\mathcal{S}(p)$ by the assumption $-1,2 \in \Sc(p)$, a contradiction. 
\end{itemize}

Let $\clubsuit\in\mathcal{T}(p)$. 
\begin{itemize}
\item If $K_3$ appears in $X''$, since $\begin{pmatrix}
K_3&\\
&\clubsuit
\end{pmatrix}\sim\begin{pmatrix}
1&&&\\
&1&&\\
&&1&\\
&&&3
\end{pmatrix}\sim
K_4$, we obtain a decomposition of $\tilde{X}$. 
\item If $K_4$ appears, since $\begin{pmatrix}
K_4&\\
&\clubsuit
\end{pmatrix}\sim
\begin{pmatrix}
1&&&&\\
&1&&&\\
&&1&&\\
&&&1&\\
&&&&1
\end{pmatrix}\sim
\begin{pmatrix}
K_2&\\
&K_3
\end{pmatrix}$, we obtain a decomposition of $\tilde{X}$. 
\item If three $K_2$'s appear, since $\begin{pmatrix}
K_2&&&\\
&K_2&&\\
&&K_2&\\
&&&\clubsuit
\end{pmatrix}\sim
\begin{pmatrix}
1&&&&&&\\
&1&&&&&\\
&&1&&&&\\
&&&1&&&\\
&&&&1&&\\
&&&&&1&\\
&&&&&&3
\end{pmatrix}\sim
\begin{pmatrix}
K_3&\\
&K_4
\end{pmatrix}$, we obtain a decomposition of $\tilde{X}$. 
\item We see that $X'' \sim \begin{pmatrix} K_2 & \\ &\clubsuit\end{pmatrix}$ never happens. 
In fact, if it happens, then $\det(X'') \in \Tc(p)$, while $\mathcal{G}_3=\{2\}\subset\mathcal{S}(p)$, a contradiction. 
By the similar reason, we see that $X'' \sim \begin{pmatrix} K_2&&\\ &K_2&\\ &&\clubsuit \end{pmatrix}$ never happens. 
\end{itemize}

\medskip

\noindent
\underline{(2) $-1\in\mathcal{S}(p)$ and $2,3\in\mathcal{T}(p)$:} 
In this case, we can straightforwardly check the following from Example~\ref{ex:det} and Corollary~\ref{kei}: 
\begin{align*}
&K_2 \sim \begin{pmatrix} 1 & \\ &1 \end{pmatrix}, \;
K_3\sim \begin{pmatrix}
1&&\\
&1&\\
&&2 
\end{pmatrix}, \; 
K_4\sim \begin{pmatrix}
1&&&\\
&1&&\\
&&1&\\
&&&2
\end{pmatrix}\text{ and }\\ 
&B \sim 
\begin{pmatrix}
1&&&&\\
&1&&&\\
&&1&&\\
&&&1&\\
&&&&1
\end{pmatrix}\text{ over }\F_p.
\end{align*}

Let $\clubsuit\in\mathcal{S}(p)$. 
\begin{itemize}
\item If $K_3$ appears in $X''$, since $\begin{pmatrix}
K_3&\\
&\clubsuit
\end{pmatrix}\sim\begin{pmatrix}
1&&&\\
&1&&\\
&&1&\\
&&&2
\end{pmatrix}\sim
K_4$, we obtain a decomposition of $\tilde{X}$. 
\item If $K_4$ appears, since $\begin{pmatrix}
K_4&\\
&\clubsuit
\end{pmatrix}\sim
\begin{pmatrix}
1&&&&\\
&1&&&\\
&&1&&\\
&&&1&\\
&&&&2\end{pmatrix}\sim\begin{pmatrix}
K_2&\\
&K_3
\end{pmatrix}$, we obtain a decomposition of $\tilde{X}$. 
\item If $B$ appears, since $\begin{pmatrix}
B&\\
&\clubsuit
\end{pmatrix}\sim\begin{pmatrix}
1&&&&&\\
&1&&&&\\
&&1&&&\\
&&&1&&\\
&&&&1&\\
&&&&&1
\end{pmatrix}\sim\begin{pmatrix}
K_3&\\
&K_3
\end{pmatrix}$, we obtain a decomposition of $\tilde{X}$. 
\item If two $K_2$'s appear, since $\begin{pmatrix}
K_2&&\\
&K_2&\\
&&\clubsuit
\end{pmatrix}\sim
\begin{pmatrix}
1&&&&\\
&1&&&\\
&&1&&\\
&&&1&\\
&&&&1
\end{pmatrix}
\sim B$, we obtain a decomposition of $\tilde{X}$. 
\item We see that $X'' \sim \begin{pmatrix} K_2 &\\ &\clubsuit\end{pmatrix}$ never happens since $\det(X'') \in \Sc(p)$ but $\Gc_3 \subset \Tc(p)$. 
\end{itemize}

Let $\clubsuit\in\mathcal{T}(p)$. 
\begin{itemize}
\item If $K_2$ appears in $X''$, since $\begin{pmatrix}
K_2& \\
&\clubsuit
\end{pmatrix}\sim
\begin{pmatrix}
1&&\\
&1&\\
&&2
\end{pmatrix}
\sim K_3$, we obtain a decomposition of $\tilde{X}$. 
\item If $K_3$ appears, since $\begin{pmatrix}
K_3&\\
&\clubsuit
\end{pmatrix}\sim\begin{pmatrix}
1&&&\\
&1&&\\
&&2&\\
&&&\clubsuit
\end{pmatrix}\sim
\begin{pmatrix}
K_2&\\
&K_2
\end{pmatrix}$, we obtain a decomposition of $\tilde{X}$. 
\item If $K_4$ appears, since $\begin{pmatrix}
K_4&\\
&\clubsuit
\end{pmatrix}\sim
\begin{pmatrix}
1&&&&\\
&1&&&\\
&&1&&\\
&&&1&\\
&&&&1
\end{pmatrix}
\sim B$, we obtain a decomposition of $\tilde{X}$. 
\item If $B$ appears, since $\begin{pmatrix}
B&\\
&\clubsuit
\end{pmatrix}\sim\begin{pmatrix}
1&&&&&\\
&1&&&&\\
&&1&&&\\
&&&1&&\\
&&&&1&\\
&&&&&\clubsuit
\end{pmatrix}\sim\begin{pmatrix}
K_2&\\
&K_4
\end{pmatrix}$, we obtain a decomposition of $\tilde{X}$. 
\end{itemize}

\medskip

\noindent
\underline{(3) $3\in\mathcal{S}(p)$ and $-1\in\mathcal{T}(p)$:} We divide the discussions into two cases; either $2 \in \Sc(p)$ or $2 \in \Tc(p)$. 

\smallskip

\noindent
(3-1): In the case $2 \in \Sc(p)$, we see from Example~\ref{ex:det} and Corollary~\ref{kei} the following: 
\begin{align*}
&K_2\sim \begin{pmatrix} 1&\\ &-1 \end{pmatrix}, \;
K_3\sim \begin{pmatrix}
1&&\\
&1&\\
&&1
\end{pmatrix}, \;
K_4\sim \begin{pmatrix}
1&&&\\
&1&&\\
&&1&\\
&&&-1
\end{pmatrix} \text{ and }\\ 
&C_5\sim \begin{pmatrix}
1&&&&\\
&1&&&\\
&&1&&\\
&&&1&\\
&&&&1
\end{pmatrix}\text{ over }\F_p.
\end{align*}

If $\clubsuit\in\mathcal{S}(p)$, then we can see the existence of a decomposition by the following computations: 
\begin{align*}
&\begin{pmatrix}
K_2&&\\ &K_2& \\ &&\clubsuit 
\end{pmatrix}\sim
C_5; \; 
\begin{pmatrix}
K_3&\\
&\clubsuit
\end{pmatrix}
\sim \begin{pmatrix}K_2 &\\ &K_2\end{pmatrix}; \\
&\begin{pmatrix}
K_4&\\
&\clubsuit
\end{pmatrix}\sim
\begin{pmatrix}K_2& \\ &K_3\end{pmatrix}; \;
\begin{pmatrix}
C_5&\\
&\clubsuit
\end{pmatrix}\sim
\begin{pmatrix}
K_3&\\
&K_3
\end{pmatrix}.
\end{align*}
Note that $X'' \sim \begin{pmatrix} K_2 & \\ &\clubsuit\end{pmatrix}$ never happens by $\Gc_3 \subset \Sc(p)$.

If $\clubsuit\in\mathcal{T}(p)$, then we can see the existence of a decomposition by the following computations: 
\begin{align*}
&\begin{pmatrix}
K_2&\\
&\clubsuit
\end{pmatrix}\sim K_3; \;
\begin{pmatrix}
K_3&\\
&\clubsuit
\end{pmatrix}\sim K_4; \;
\begin{pmatrix}
K_4&\\
&\clubsuit
\end{pmatrix}\sim
C_5; \;
\begin{pmatrix}
C_5&\\
&\clubsuit
\end{pmatrix}\sim
\begin{pmatrix}
K_2&&\\
&K_2&\\
&&K_2
\end{pmatrix}. 
\end{align*}

\smallskip

\noindent
(3-2): In the case $2 \in \Tc(p)$, we see the following: 
\begin{align*}
&K_2 \sim\begin{pmatrix}
1&\\
&-1
\end{pmatrix}, \;
K_3 \sim\begin{pmatrix}
1&&\\
&1&\\
&&-1
\end{pmatrix}, \;
K_4 \sim\begin{pmatrix}
1&&&\\
&1&&\\
&&1&\\
&&&-1
\end{pmatrix} \text{ and }\\ 
&C_5 \sim\begin{pmatrix}
1&&&&\\
&1&&&\\
&&1&&\\
&&&1&\\
&&&&-1
\end{pmatrix} \text{ over }\F_p.
\end{align*}

If $\clubsuit\in\mathcal{S}(p)$, then we can see the existence of a decomposition by the following computations: 
\begin{align*}
&\begin{pmatrix}
K_2& \\ &\clubsuit 
\end{pmatrix}\sim
K_3; \; 
\begin{pmatrix}
K_3&\\
&\clubsuit
\end{pmatrix}
\sim K_4; \;
\begin{pmatrix}
K_4&\\
&\clubsuit
\end{pmatrix}\sim
C_5; \\
&\begin{pmatrix}
C_5&\\
&\clubsuit
\end{pmatrix}\sim
\begin{pmatrix}
1&&&&&\\
&1&&&&\\
&&1&&&\\
&&&-1&&\\
&&&&-1&\\
&&&&&-1
\end{pmatrix}\sim
\begin{pmatrix}
K_2&&\\
&K_2&\\
&&K_2
\end{pmatrix}.
\end{align*}

If $\clubsuit\in\mathcal{T}(p)$, then we can see the existence of a decomposition by the following computations: 
\begin{align*}
&\begin{pmatrix}
K_2&&\\
&K_2&\\
&&\clubsuit
\end{pmatrix}\sim C_5; \;
\begin{pmatrix}
K_3&\\
&\clubsuit
\end{pmatrix}\sim \begin{pmatrix}K_2 &\\ &K_2\end{pmatrix}; \\
&\begin{pmatrix}
K_4&\\
&\clubsuit
\end{pmatrix}\sim
\begin{pmatrix}K_2 & \\ &K_3\end{pmatrix}; \;
\begin{pmatrix}
C_5&\\
&\clubsuit
\end{pmatrix}\sim
\begin{pmatrix}
K_3&\\
&K_3
\end{pmatrix}. 
\end{align*}
Note that $X'' \sim \begin{pmatrix} K_2 & \\ &\clubsuit\end{pmatrix}$ never happens by $\Gc_3 \subset \Tc(p)$. 

\noindent
\underline{(4) $-1,3 \in\mathcal{T}(p)$ with $p \geq 5$ or $p=3$:} We divide the discussions into two cases; either $2 \in \Sc(p)$ or $2 \in \Tc(p)$. 
Note that the discussion for the case $p=3$ is the same as $2 \in \Tc(p)$. 

\smallskip

\noindent
(4-1): In the case $2 \in \Sc(p)$, we see the following: 
\begin{align*}
K_2\sim \begin{pmatrix}
1&\\
&-1
\end{pmatrix}, \;
K_3\sim \begin{pmatrix}
1&&\\
&1&\\
&&1
\end{pmatrix}\text{ and }
C_5\sim \begin{pmatrix}
1&&&&\\
&1&&&\\
&&1&&\\
&&&1&\\
&&&&1
\end{pmatrix}\text{ over }\F_p.
\end{align*}

If $\clubsuit\in\mathcal{S}(p)$, then we can see the existence of a decomposition by the following computations: 
\begin{align*}
&\begin{pmatrix}
K_2&& \\ &K_2& \\ &&\clubsuit 
\end{pmatrix}\sim
C_5; \; 
\begin{pmatrix}
K_3&\\
&\clubsuit
\end{pmatrix}
\sim \begin{pmatrix}K_2 &\\ &K_2\end{pmatrix}; \;
\begin{pmatrix}
C_5&\\
&\clubsuit
\end{pmatrix}\sim
\begin{pmatrix}
K_3&\\
&K_3
\end{pmatrix}.
\end{align*}
Note that $X'' \sim \begin{pmatrix} K_2 & \\ &\clubsuit\end{pmatrix}$ never happens by $\Gc_3 \subset \Sc(p)$.

If $\clubsuit\in\mathcal{T}(p)$, then we can see the existence of a decomposition by the following computations: 
\begin{align*}
\begin{pmatrix}
K_2&\\
&\clubsuit
\end{pmatrix}\sim K_3; \;
\begin{pmatrix}
K_3&&\\
&K_3&\\
&&\clubsuit
\end{pmatrix}\sim \begin{pmatrix}K_2 &\\ &C_5\end{pmatrix}; \;
\begin{pmatrix}
C_5&\\
&\clubsuit
\end{pmatrix}\sim
\begin{pmatrix}K_2 & & \\ &K_2 & \\ &&K_2\end{pmatrix}. 
\end{align*}
Note that $X'' \sim \begin{pmatrix} K_3 & \\ &\clubsuit\end{pmatrix}$ never happens by $\Gc_4 \subset \Sc(p)$. 

\smallskip

\noindent
(4-2): In the case $2 \in \Tc(p)$, we see the following: 
\begin{align*}
&K_2\sim \begin{pmatrix}
1&\\
&-1
\end{pmatrix}, \;
K_3\sim \begin{pmatrix}
1&&\\
&1&\\
&&-1
\end{pmatrix}, \;
C_5\sim \begin{pmatrix}
1&&&&\\
&1&&&\\
&&1&&\\
&&&1&\\
&&&&-1
\end{pmatrix}.
\end{align*}

If $\clubsuit\in\mathcal{S}(p)$, then we can see the existence of a decomposition by the following computations: 
\begin{align*}
&\begin{pmatrix}
K_2& \\ &\clubsuit 
\end{pmatrix}\sim
K_3; \; 
\begin{pmatrix}
K_3&&\\
&K_3&\\
&&\clubsuit
\end{pmatrix}
\sim \begin{pmatrix}K_2 &\\ &C_5\end{pmatrix}; \;
\begin{pmatrix}
C_5&\\
&\clubsuit
\end{pmatrix}\sim
\begin{pmatrix}
K_2&&\\ 
&K_2& \\
&&K_2
\end{pmatrix}.
\end{align*}
Note that $X'' \sim \begin{pmatrix} K_3 & \\ &\clubsuit\end{pmatrix}$ never happens by $\Gc_4=\{-3,1\} \subset \Sc(p)$.

If $\clubsuit\in\mathcal{T}(p)$, then we can see the existence of a decomposition by the following computations: 
\begin{align*}
\begin{pmatrix}
K_2&&\\
&K_2&\\
&&\clubsuit
\end{pmatrix}\sim C_5; \;
\begin{pmatrix}
K_3&\\
&\clubsuit
\end{pmatrix}\sim \begin{pmatrix}K_2 &\\ &K_2\end{pmatrix}; \;
\begin{pmatrix}
C_5&\\
&\clubsuit
\end{pmatrix}\sim
\begin{pmatrix}K_3 & \\ &K_3 \\ \end{pmatrix}. 
\end{align*}
Note that $X'' \sim \begin{pmatrix} K_2 & \\ &\clubsuit\end{pmatrix}$ never happens by $\Gc_3 \subset \Tc(p)$.

\bigskip

Therefore, the proof of Theorem \ref{mainthm} is completed. \qed

\bigskip

\section{Proof of Theorem~\ref{only2int}}\label{sec:only2int}

This section is devoted to proving Theorem~\ref{only2int}. 

In the case where $-1,3 \in \Sc(p)$ and $2 \in \Tc(p)$, we see the following: 
\begin{align*}
&K_2 \sim \begin{pmatrix}1 & \\ &1 \end{pmatrix}, \; 
K_3 \sim \begin{pmatrix}1 & & \\  &1 & \\  & &2 \end{pmatrix}, \; 
B\sim \begin{pmatrix}1 & &&& \\ &1&&& \\ &&1&& \\ &&&1& \\ &&&&1 \end{pmatrix}, \\
&K_6 \sim \begin{pmatrix}1 & &&&& \\ &1&&&& \\ &&1&&& \\ &&&1&& \\ &&&&1& \\ &&&&&5 \end{pmatrix} \underset{\text{if }5 \in \Tc(p)}{\sim} 
\begin{pmatrix}1 & &&&& \\ &1&&&& \\ &&1&&& \\ &&&1&& \\ &&&&1& \\ &&&&&2 \end{pmatrix}\text{ and } \\ 
&D \sim \begin{pmatrix}1 & &&&& \\ &1&&&& \\ &&1&&& \\ &&&1&& \\ &&&&1& \\ &&&&&7 \end{pmatrix} \underset{\text{if }7 \in \Tc(p)}{\sim} 
\begin{pmatrix}1 & &&&& \\ &1&&&& \\ &&1&&& \\ &&&1&& \\ &&&&1& \\ &&&&&2 \end{pmatrix}\text{ over }\F_p. 
\end{align*}

We proceed the proof in the same way and work with the same notation as that of Theorem~\ref{mainthm}. 
 
If $\clubsuit \in \Sc(p)$, then we see the existence of a decomposition by the following computations: 
\begin{align*}
&\begin{pmatrix}
K_2&& \\ &K_2& \\ &&\clubsuit 
\end{pmatrix}\sim
B; \; 
\begin{pmatrix}
K_3&&\\
&K_3&\\
&&\clubsuit
\end{pmatrix}
\sim \begin{pmatrix}K_2 &\\ &B \end{pmatrix}; \; 
\begin{pmatrix}
B&\\
&\clubsuit
\end{pmatrix}\sim \begin{pmatrix}K_3 &\\ &K_3\end{pmatrix}; \\
&\begin{pmatrix}
K_2&& \\ &K_3&  \\ &&\clubsuit
\end{pmatrix}\sim \begin{cases} K_6 \;&\text{if }5 \in \Tc(p), \\ D &\text{if }7 \in \Tc(p), 
\end{cases} \\
&\begin{pmatrix}
K_6 \text{ or }D& \\
&\clubsuit
\end{pmatrix} \sim \begin{pmatrix} K_2&& \\ &K_2& \\ &&K_3 \end{pmatrix} \text{ if $5 \in \Tc(p)$ or $7 \in \Tc(p)$, respectively}. 
\end{align*}
Note that $X'' \sim \begin{pmatrix} K_2 & \\ &\clubsuit \end{pmatrix}$ and $X'' \sim \begin{pmatrix} K_3 & \\ &\clubsuit \end{pmatrix}$ 
never happen by $\Gc_3 \subset \Tc(p)$ and $\Gc_4 \subset \Sc(p)$, respectively, 
and $\begin{pmatrix} K_2 & & \\ &K_3 & \\ &&\clubsuit \end{pmatrix}$ never happens 
by $\Gc_6=\{-5,-4,-1,3,4,7\} \subset \Sc(p)$ when $5 \in \Sc(p)$ and $7 \in \Sc(p)$.   

If $\clubsuit \in \Tc(p)$, then we see the existence of a decomposition by the following computations: 
\begin{align*}
&\begin{pmatrix}
K_2& \\ &\clubsuit 
\end{pmatrix}\sim
K_3; \; 
\begin{pmatrix}
K_3& \\
&\clubsuit
\end{pmatrix}
\sim \begin{pmatrix}K_2 &\\ &K_2 \end{pmatrix}; \;
\begin{pmatrix}
B&\\
&\clubsuit
\end{pmatrix}\sim \begin{cases} K_6 \;&\text{if }5 \in \Tc(p), \\ D &\text{if }7 \in \Tc(p), \end{cases} \\
&\begin{pmatrix}
K_6 \text{ or }D& \\
&\clubsuit
\end{pmatrix} \sim \begin{pmatrix} K_2& \\ &B \end{pmatrix} \text{ if $5 \in \Tc(p)$ or $7 \in \Tc(p)$, respectively}. 
\end{align*}
Note that $X'' \sim \begin{pmatrix} B & \\ &\clubsuit \end{pmatrix}$ never happens 
when $5 \in \Sc(p)$ and $7 \in \Sc(p)$ by $\Gc_6 \subset \Sc(p)$. \qed

\bigskip

\section{Proof of Theorem~\ref{allins}}\label{sec:allins}

This section is devoted to proving Theorem~\ref{allins}.

\begin{lem}\label{easy} 
(1) For any integer $n$, there exists a graph $G$ such that $\det G=n$. \\
(2) Fix a prime $p$. Then there exist infinitely many graphs whose determinant belong to $\Tc(p)$. 
\end{lem}
\begin{proof}
(1) Since $\det K_{n+1} = (-1)^n n$ for any $n \geq 2$ and $\det K_2=-1$, we may take $G=K_{n+1}$ or $G=K_{n+1} \oplus K_2$. 

\noindent
(2) Let $G$ be a graph with $\det G \in \Tc(p)$. Note that the existence of such graph is guaranteed by the previous statement (1). 
Then, for any $m \geq 0$, we have $\det((2m+1)G)=(\det G)^{2m+1} \in \Tc$. 
\end{proof}

By Lemma~\ref{easy} (2), we can define the following invariant $N(p)$ and $N'(p)$ for a given prime $p$: 
$$N(p)=\min\{ n : \text{there is a graph $G \in \textbf{Graph}(n)$ with $\det G \in \Tc(p)$}\}$$
and 
\begin{align*}
N'(p)=\min\{ n : \text{there is a graph $G \in \textbf{Graph}(n)$ with }&\text{$\det G \in \Tc(p)$} \\
&\text{and $n>N(p)$}\}.
\end{align*}

Fix a prime $p$ and let $N=N(p)$ and $N'=N'(p)$. 
Take $X_4 \in \textbf{Graph}(N)$ (resp. $X_5  \in \textbf{Graph}(N')$) with $\det X_4 \in \Tc(p)$ (resp. $\det X_5 \in \Tc(p)$). 
Note that $\det G \in \Sc(p)$ holds for any $G \in \textbf{Graph}(n)$ if $N<n<N'$ by definitions. 
In what follows, we prove Theorem~\ref{allins} by using these $X_4$ and $X_5$. 

In the case where $-1,2,3 \in \Sc(p)$, we see the following: 
\begin{align*}
&K_2 \sim \begin{pmatrix}1 & \\ &1 \end{pmatrix}, \; 
K_3 \sim \begin{pmatrix}1 & & \\  &1 & \\  & &1 \end{pmatrix}, \;
X_4 \sim \begin{pmatrix}
1 & & &\\
 &\ddots & & \\
 & &1 & \\
 & & &x
\end{pmatrix}\text{ and}\\
&X_5 \sim \begin{pmatrix}
1 & & & &\\
 &1 & & & \\
 & &\ddots & & \\
 & & &1 & \\
 & & & &x
\end{pmatrix}\text{ over }\F_p, 
\end{align*}
where $x$ is an element of $\Tc(p)$. 
We proceed the proof in the same way and work with the same notation as before.

Let $\clubsuit \in \Sc(p)$. 
Then we see the existence of a decomposition by the following computations: 
\begin{align*}
&\begin{pmatrix}
K_2& \\ &\clubsuit 
\end{pmatrix}\sim
K_3; \; 
\begin{pmatrix}
K_3&\\
&\clubsuit
\end{pmatrix}
\sim \begin{pmatrix}K_2 &\\ &K_2 \end{pmatrix}. 
\end{align*}

If $N'=N+1$, then we also see the following: 
\begin{align*}
\begin{pmatrix}
X_4&\\
&\clubsuit
\end{pmatrix}\sim X_5, \; 
\begin{pmatrix}
X_5 &\\
&\clubsuit
\end{pmatrix} \sim \begin{pmatrix} K_2& \\ &X_4 \end{pmatrix}. 
\end{align*}
If $N'>N+1$, then we see that $X'' \sim \begin{pmatrix} X_4 & \\ &\clubsuit\end{pmatrix}$ never happens since $\det G \in \Sc(p)$ for any $G \in \textbf{Graph}(N+1)$. 
Moreover, since $N'-N \geq 2$, it is easy to see that there are nonnegative integers $a,b$ with $N'+1-N = 2a+3b$. Hence, 
$\begin{pmatrix} X_5 & \\ &\clubsuit \end{pmatrix} \sim aK_2 \oplus bK_3 \oplus X_4$ holds.  

Let $\clubsuit \in \Tc(p)$. Let $n$ be the size of $X''$, i.e., the number of its vertices. 

If $\det X'' \in \Sc(p)$, since $\det \begin{pmatrix} X'' & \\ &\clubsuit \end{pmatrix} \in \Tc(p)$, 
we see from the definitions of $N$ and $N'$ that $n+1 = N$ or $n+1 \geq N'$. 
\begin{itemize}
\item If $n+1=N$, then $\begin{pmatrix} X'' & \\ &\clubsuit \end{pmatrix} \sim X_4$. 
\item If $n+1=N'$, then $\begin{pmatrix} X'' & \\ &\clubsuit \end{pmatrix} \sim X_5$. 
\item If $n+1 > N'$, since $n+1 - N \geq 2$, by using nonnegative integers $a,b$ with $(n+1)-N =2a+3b$, 
we obtain that $\begin{pmatrix} X'' & \\ &\clubsuit \end{pmatrix} \sim aK_2 \oplus bK_3 \oplus X_4$. 
\end{itemize}

If $\det X'' \in \Tc(p)$, since $\det \begin{pmatrix} X'' & \\ &\clubsuit \end{pmatrix} \in \Sc(p)$, 
there exist nonnegative integers $a,b$ such that $\begin{pmatrix} X'' & \\ &\clubsuit \end{pmatrix} \sim aK_2 \oplus bK_3$. \qed

\end{document}